\documentclass{amsart}
\usepackage{salch, xy}
\usepackage{rotating}
\xyoption{all}
\title{Grothendieck duality under $\Spec \mathbb{Z}$.}
\date{August 2010.}
\begin{document}
\begin{abstract}
We define the derived category of a concrete category in a way
which extends the usual definition of the derived category of a ring, and
we prove that the bounded-below derived category of 
$\Spec \mathbb{M}_0$
(an approximation, used by e.g. Connes and Consani, to
``$\Spec$ of the field with one element'') is the stable homotopy category of 
connective spectra. We also describe some basic features of Grothendieck duality for the
map from $\Spec \mathbb{Z}$ to $\Spec \mathbb{M}_0$, or, what comes to the same thing,
the map from $\Spec \mathbb{Z}$ to $\Spec$ of the sphere spectrum; these
basic features include a computation of
the homology of the dualizing complex $f^!(S)$ of abelian groups associated to 
the sphere spectrum.
\end{abstract}
\maketitle

\section{Introduction.}

It has been known, at least since the time of Waldhausen coining the phrase 
``brave
new ring,'' that essentially all of commutative algebra embeds into stable 
homotopy
theory via the Eilenberg-Maclane functor $H$; an excellent reference for this
idea, which also makes clear the kind of technical obstacles involved in making it
precise and proving it, is \cite{MR2306038}. From this point of view,
$E_\infty$-ring spectra are a generalization of commutative rings, and one wants
to be able to knit together $E_\infty$-ring spectra to form some geometric
objects, in the same way that commutative rings are knitted together to form
schemes; such geometric objects would bring together the methods and results of
algebraic geometry with the methods and results of stable homotopy theory. Two
multi-volume works (Toen and Vezzosi's HAG series \cite{MR2137288} and 
\cite{MR2394633}, and
the DAG volumes by Lurie) have already been written, establishing
the basic definitions and results to allow us to work with these 
``derived schemes.''

One way of viewing the situation is that all of algebraic geometry takes place, as
is well-known, over the base scheme $\Spec \mathbb{Z}$, but stable homotopy theory
takes place over the base scheme ``$\Spec S$,'' which we have to put in quote 
marks,
since it is not actually a scheme and not actually $\Spec$ of a commutative ring
in the classical sense; here $S$ is the sphere spectrum, the unit object of the
tensor product (smash product) on the category of spectra. In this respect,
stable homotopy theory is a kind of algebraic geometry which happens 
``under $\Spec \mathbb{Z}$.'' Actually, it is not the {\em only} kind of 
algebraic 
geometry ``under $\Spec \mathbb{Z}$,'' as the various candidates for what should 
be
considered ``schemes over $\Spec \mathbb{F}_1$'' also sit under 
$\Spec \mathbb{Z}$;
see \cite{MR2507727} for a useful perspective on this (although it should
be remarked that the version of $\Spec \mathbb{F}_1$ in that paper does not
have all the properties that non-commutative geometers want to have in a
good construction of $\Spec \mathbb{F}_1$, and there are other candidates, 
e.g. Connes' and Consani's $\mathbb{F}_1$-schemes from \cite{schemesoverf1}, 
or Borger's schemes built from $\lambda$-rings, which are preferable to
those in \cite{MR2507727} for various reasons). One approximation
to the category of $\mathbb{F}_1$-schemes is the category
of $\mathbb{M}_0$-schemes, which sits under $\Spec \mathbb{Z}$,
and which we will be interested in, in this note.

Jack Morava has advocated the study of the morphism $\Spec \mathbb{Z}\rightarrow
\Spec S$, and in particular, he has encouraged the development and application of
the methods and results from Hartshorne's \cite{MR0222093}
to this morphism. From e.g. \cite{MR0222093} one knows that, given
a scheme $X$ of finite type over a field $k$, one has the
scheme morphism $X\stackrel{f}{\longrightarrow} \Spec k$, and
one has a kind of duality functor on the derived category of
quasicoherent $\mathcal{O}_X$-modules, represented by 
a chain complex $f^!k[0]$; here $k[0]$ is the chain complex of 
$k$-modules consisting of $k$ in degree $0$ and the trivial
module in all other degrees, and $f^!$ is a derived right
adjoint to $Rf_*$, the total right derived functor of the 
push-forward of modules along $f$. Clearly, the functor
$\hom_{\Mod(k)}(-,k)$ is a duality functor, in the most 
straightforward sense, on the category of finite-dimensional
$k$-vector spaces; so $f^!k[0]$ is 
``pulling back'' the representing object $k$ for this
duality on $k$-modules to a representing object for a
(derived) duality on (chain complexes of) $\mathcal{O}_X$-modules.

Now we want to do the same thing, but for the morphism
$\Spec \mathbb{Z}\stackrel{f}{\longrightarrow} \Spec S$. 
One knows that the sphere spectrum $S$ represents the 
Spanier-Whitehead duality functor, and one would like to know what kind of 
``dualizing complex'' of abelian groups one gets when one applies $f^!$ to the
sphere spectrum, where $f^!$ is a derived right adjoint for $Rf_*$.
In this note, we accomplish at least a large part of that task:
we make precise 
(in two different, non-equivalent ways, each of which has some desirable properties)
what we mean when we speak of 
``the morphism $\Spec\mathbb{Z}\stackrel{f}{\longrightarrow} \Spec S$,'' 
we describe the functors $f_*, f^*$, and $f^!$ associated to $f$, and we compute the
homology groups of the dualizing complex $f^!S$, as well as some interesting
properties of this dualizing complex. All of this essentially follows as consequences
of results proved by others.

There is another interpretation of our results which is worth
considering. One useful approximation to the category of 
$\mathbb{F}_1$-schemes is the category of $\mathbb{M}_0$-schemes
(see e.g. \cite{schemesoverf1} for definitions and basic
properties of $\mathbb{M}_0$-schemes, as well as their role
in the construction of $\mathbb{F}_1$-schemes): 
these objects resemble
schemes but, by design, are given Zariski-locally by 
commutative monoids with zero elements, rather than by commutative
rings. As we show in this note, the derived category of
$\mathbb{M}_0$ is equivalent (as a triangulated category) to the
stable homotopy category of spectra;
and the description we obtain of $f^!S$, including its homology
$H_*(f^!S)$, is also a description of $g^!\mathbb{M}_0[0]$ and
its homology $H_*(g^!\mathbb{M}_0[0])$, where $f$ is the
morphism $\Spec \mathbb{Z}\stackrel{f}{\longrightarrow}
\Spec S$ and $g$ is the morphism $\Spec \mathbb{Z}
\stackrel{g}{\longrightarrow} \Spec \mathbb{M}_0$.

Our description of $H_*(f^!S)$ bears a curious relation to the ideal class group
from number theory. Let $K/\mathbb{Q}$ be a finite field extension; then 
the ideal class group $\Cl(K)$ of $K$ admits the following description as a double quotient:
\[ \Cl(K)\cong G(\mathbb{A}_K^{\infty})\backslash G(\mathbb{A}_K)/G(K),\]
where $G$ is the multiplicative group scheme, $\mathbb{A}_K$ is the adele ring of 
$K$, and $\mathbb{A}_K^{\infty}$ is the subring of $\mathbb{A}_K$ consisting of the infinite adeles
(i.e., the restricted direct product of the completions of $K$ at Archimedean places).
Let $M\mathcal{O}_K$ be the
Moore spectrum of the ring of integers $\mathcal{O}_K$ of $K$; the sphere
spectrum $S$ is the special case $S\simeq M\mathcal{O}_{\mathbb{Q}}$. We prove that one has the isomorphism
\[ H_{-1}(f^!(M\mathcal{O}_K)) \cong G(\mathbb{A}_K^{\infty})\backslash G(\mathbb{A}_K)/G(K),\]
where now $G$ is the {\em additive}, rather than multiplicative, group scheme. This isomorphism is natural in the choice of number
field $K$. In this sense, the homology of the ``dualizing complex''
$f^!(M\mathcal{O}_K)$ is a kind of additive analogue of the ideal class group of $K$. We also show that
the homology groups of $f^!(M\mathcal{O}_K)$ are trivial in all degrees $\neq -1$.

Everything
we do in this note follows from theorems proved by others; we organize those
results in a way that provides an answer to the questions asked by J. Morava, and,
we think, sheds a little bit more light on the relations between algebraic geometry, 
stable homotopy theory, and geometry over $\mathbb{F}_1$.

\section{$\mathbb{M}_0$ and spectra.}

\begin{prop}
There exists a functor $\Rings\stackrel{\beta^*}{\longrightarrow} \Monoids_0$, from the category of
commutative rings to the category of commutative monoids equipped with a zero element; this 
functor $\beta^*$ simply forgets the addition operation on the ring. This functor also has
a left adjoint $\beta$, which sends a commutative monoid $M$ with zero element ${\bf 0}$ to the monoid
ring $\mathbb{Z}[M]/({\bf 0})$.
Let $p\Sets$ be the closed symmetric monoidal category whose objects are pointed sets and whose
monoidal product is given by the smash product
\[ X\wedge Y \cong (X\times Y)/(X\vee Y),\]
the Cartesian product of $X$ and $Y$ with the one-point union of $X$ and $Y$ collapsed to the basepoint. 
Let $\Ab$ be the closed symmetric monoidal category whose objects are abelian groups and whose
monoidal product is given by the tensor product (over $\mathbb{Z}$). 
Let $\Ab\stackrel{\alpha^*}{\longrightarrow}p\Sets$ be the forgetful functor and let
$p\Sets\stackrel{\alpha}{\longrightarrow}\Ab$ be its left adjoint, which sends a pointed set $(S,*)$ to the free abelian
group $\left(\oplus_{s\in S}\mathbb{Z}\right)/(*=0)$. Then, on applying $\Comm$ (i.e., passing to the
categories of commutative monoid objects), we get the commutative
diagram of categories and functors
\[\label{categories diagram}\xymatrix{ \Comm(p\Sets)\ar[r]^{\Comm(\alpha)}\ar[d]^{\simeq} &
 \Comm(\Ab)\ar[r]^{\Comm(\alpha^*)}\ar[d]^{\simeq} &
 \Comm(p\Sets)\ar[d]^{\simeq} \\
\Monoids_0\ar[r]^{\beta} & \Rings\ar[r]^{\beta^*} & \Monoids_0}\]
where the functors marked $\simeq$ are equivalences of categories.
\end{prop}
\begin{proof}
Let $M$ be a commutative monoid with zero, with multiplication map $M\times M\stackrel{\nabla}{\longrightarrow} M$. 
Then the multiplication map
factors uniquely through the smash product:
\[\xymatrix{ M\times M\ar@/_10pt/[rr]_{\nabla} \ar[r] & 
(M\times M)/(M\vee M)\ar[r] & M,}\]
since $0\bullet m = m\bullet 0 = 0$ for any $m\in M$. So to any commutative monoid with zero we associate a commutative monoid
in $p\Sets$; this construction is natural in $M$. Given a commutative monoid $N\wedge N\stackrel{\nabla}{\longrightarrow} N$ 
in $p\Sets$, one has the structure of a commutative monoid with zero on $N$, given by letting $n\bullet n^{\prime}$ be the image 
of $(n,n^{\prime})$ under the composite
\[\xymatrix{ N\times N\ar[r] & N\wedge N\ar[r]^{\nabla} & N.}\]
This gives a monoid operation on $N$ that only fails to be defined on pairs of the form $(0,n)$ and $(n,0)$; for all such
pairs we set $n\bullet 0 = 0\bullet n = 0$. Hence the leftmost vertical map in diagram~\ref{categories diagram}
is an equivalence of categories. That the central vertical map in diagram~\ref{categories diagram} is an equivalence
of categories is well-known.

That $\Comm(\alpha^*)$ is naturally isomorphic to the functor $\beta^*$ follows immediately from inspection---they are both
forgetful functors. If $M$ is a commutative monoid with zero, then the underlying abelian group of $\beta(M)$ is 
$\alpha(M)$; so $\Comm(\alpha^*)$ is naturally isomorphic to $\beta^*$.
\end{proof}

\begin{lemma} \label{psset is spset}
{\bf (Simplicial $\mathbb{M}_0$-modules ``are'' pointed topological spaces.)} 
Let $ps\Sets$ be the category of pointed simplicial sets (i.e., simplicial
sets $X$ equipped with a choice of point in $X[0]$), and let $sp\Sets$ be the
category of simplicial pointed sets (i.e., simplicial objects in the
category of pointed sets). Then the forgetful functor
\[ sp\Sets\rightarrow ps\Sets\]
is an equivalence of categories.
\end{lemma}
\begin{proof}
A quasi-inverse for the forgetful functor is the functor sending
a pointed simplicial set $X$ with point $x\in X[0]$ to the simplicial pointed
set whose basepoint in $X[n]$ is $(\sigma_0\circ \dots \circ \sigma_0)(x)$, the
$n$-fold composite of the zeroth degeneracy maps applied to the basepoint
in $X[0]$. All that needs to be checked is that this actually defines
a simplicial pointed set, i.e., that the face and degeneracy maps
respect the basepoints. Let $x_n$ denote the basepoint 
$(\sigma_0\circ \dots \circ\sigma_0)(x) \in X[n]$, and suppose
that $\delta_jx_i = x_{i-1}$ for all $j$ and all $i<n$. Then
\begin{eqnarray*} \delta_0x_n &  = &  (\delta_0\circ \sigma_0)(x_{n-1}) \\
 & = & x_{n-1} \\
 & = & (\delta_1\circ \sigma_0)(x_{n-1}) \\
 & = & \delta_1x_n, \mbox{\ and} \\
\delta_jx_n & = & (\delta_j\circ \sigma_0)(x_{n-1}) \\
 & = & (\sigma_0\circ \delta_{j-1})(x_{n-1}) \\
 & = & \sigma_0(x_{n-2}) \\
 & = & x_{n-1},\end{eqnarray*}
for any $j>1$, by the simplicial identities and our hypotheses; by induction,
the face maps $\delta_i$ respect the basepoints $\{ x_n\}$. Now suppose that
$\sigma_jx_i = x_{i-1}$ for all $j$ and all $i<n$. Then 
\begin{eqnarray*}  \sigma_jx_n & = & \sigma_j\sigma_0x_{n-1} \\
 & = & \sigma_0\sigma_{j-1}x_{n-1} \\
 & = & \sigma_0\sigma_0\sigma_{j-2} x_{n-2} \\
 & = & \dots \\
 & = & (\sigma_0 \dots \sigma_0)(x_{n-j}) \\
 & = & x_{n+1} \end{eqnarray*}
for any $j>0$, by the simplicial identities and our hypotheses; so
the degeneracy maps $\sigma_i$ respect the basepoints $\{ x_n\}$.
\end{proof}

Now we will begin to speak of derived categories; in this section,
we adopt the convention that all our
complexes are chain complexes rather than cochain complexes, i.e., the differentials 
are of the form $C_i\rightarrow C_{i-1}$; they lower degree, rather than raising degree. 
By a ``bounded-below'' chain complex we mean 
a chain complex $C_\cdot$ such that there exists an integer $M$ such that 
$C_m\cong 0$ for all $m < M$.

We would like to say something meaningful about the derived
category $D(\mathbb{M}_0)$ of complexes
of $\mathbb{M}_0$-modules. We can't speak
of the category of chain complexes of $\mathbb{M}_0$-modules,
since $\mathbb{M}_0$-modules are simply pointed sets, and the 
category $p\Sets$ is not an abelian category; but 
by the Dold-Kan theorem, whenever $\mathcal{C}$
is an abelian category, we have an Quillen equivalence 
between chain complexes, concentrated in nonnegative degrees,
of objects in $\mathcal{C}$,  with
the projective model structure,
and simplicial objects in $\mathcal{C}$ with model structure
coming from the standard model structure on simplicial sets. So we want
a description of the operation that takes a commutative ring $R$ to the
derived category $D(R)$ which does not use the fact that the category of
$R$-modules is abelian, but perhaps uses the model structure on simplicial
$R$-modules instead.

Here is how we arrive at such a description, which is largely inspired by Lurie's treatment of
derived categories in \cite{dagi} (although this description does not use Lurie's 
$\infty$-categorical technology): 
let $\mathcal{C}$ be any concrete category (i.e., a category equipped with a
forgetful functor to the category of sets), not necessarily an abelian category. 
Then we have a model structure on $s\mathcal{C}$, the 
category of simplicial objects in $\mathcal{C}$, given by letting a morphism in $s\mathcal{C}$
be a fibration if the underlying map of simplicial sets is a fibration, and likewise for
weak equivalences and cofibrations. Then let $S^+(s\mathcal{C})$ be the
category of {\em bounded-below spectrum objects in $s\mathcal{C}$}: this is the category of 
functors $\mathbb{Z}\times\mathbb{Z}\stackrel{F}{\longrightarrow} s\mathcal{C}$ 
satisfying the following conditions:
\begin{enumerate}
\item $F$ preserves finite homotopy limits,
\item for any object $(m,n)\in\mathbb{Z}\times\mathbb{Z}$ with $m\neq n$, the object $F(m,n)$ is
weakly equivalent to a final object of $s\mathcal{C}$, 
\item for any object $(m,m)\in \mathbb{Z}\times\mathbb{Z}$, the object $F(m,m)$ is a fibrant
object of $s\mathcal{C}$, and
\item there exists some integer $M$ such that $F(m,m)$ is weakly equivalent to a final object
of $s\mathcal{C}$ for every $m<M$.
\end{enumerate}
Here by $\mathbb{Z}$ we mean the integers with their usual partial ordering, considered as a category.
The practical significance of this definition is that any such functor $F$ specifies a homotopy 
pullback square 
\[ \xymatrix{ F(m,m) \ar[r]\ar[d] & \pt \ar[d] \\ \pt\ar[r] & F(m+1,m+1) }\]
for each $m\in\mathbb{Z}$, i.e., a choice of weak equivalence $F(m,m)\simeq \Omega F(m+1,m+1)$,
and $F(m,m)$ is trivial for $m << 0$; when $\mathcal{C}$ is the category of pointed sets, after passage to
the homotopy category, this definition
recovers the typical homotopy category of bounded-below $\Omega$-spectra.
On the other hand, when $\mathcal{C}$ is the category of $R$-modules, where $R$ is some commutative ring, then
the model category $s\mathcal{C}$ is (by the Dold-Kan theorem) equivalent to the category of 
chain complexes of $R$-modules concentrated in nonnegative degrees, with the projective model structure;
fibrant objects in this model category are those in which every $R$-module in the complex is a projective
$R$-module, while weak equivalences are quasiisomorphisms, 
so the category $S^+(s\mathcal{C})$ of bounded-below spectra in $R$-modules is, after passage to
the homotopy category, the usual bounded-below derived category of $R$-modules; that is,
$\Ho S^+(s\Mod(R))\simeq D^+(R)$. Hence, if $\mathcal{C}$ is a concrete category, we regard
$\Ho S^+(s\mathcal{C})$ as the {\em bounded-below derived category of $\mathcal{C}$}, and we
may write $D^+(\mathcal{C})$ for it.
If $\mathcal{C}$ is a category of modules over some ring- or monoid-like 
object $M$ then we may also write $D^+(M)$ for $D^+(\mathcal{C})$, a useful abuse of notation.

\begin{prop} The bounded-below derived category $D^+(\mathbb{M}_0)$ of 
$\mathbb{M}_0$-modules is equivalent
to the stable homotopy category of bounded-below spectra.
\end{prop}
\begin{proof}
The category of $\mathbb{M}_0$-modules is precisely the category $p\Sets$ of pointed sets.
The category $sp\Sets$ of simplicial pointed sets is equivalent to the category 
$ps\Sets$ of pointed simplicial sets, by Lemma~\ref{psset is spset}, and this equivalence
is easily seen to preserve the model structures; so $D^+(\mathbb{M}_0)
\simeq \Ho S^+(ps\Sets)$, the homotopy category of bounded-below
spectrum objects in pointed Kan complexes, which is equivalent to the
homotopy category of bounded-below $\Omega$-spectra.
\end{proof}

We also want to describe the effects of the functors 
$\Ab\stackrel{\alpha^*}{\longrightarrow}p\Sets$ and
$p\Sets\stackrel{\alpha}{\longrightarrow}\Ab$
on the level of derived categories. These two functors jointly 
describe a morphism 
\[ \Spec \mathbb{Z}\stackrel{g}{\longrightarrow}\Spec \mathbb{M}_0\]
such that $g_* = \alpha^*$ and $g^*=\alpha$; we want to understand the morphisms
$D^+(\mathbb{Z})\stackrel{g_*}{\longrightarrow} D^+(\mathbb{M}_0)$
and
$D^+(\mathbb{M}_0)\stackrel{g^*}{\longrightarrow} D^+(\mathbb{Z})$
induced on the derived categories. Taking a simplicial pointed set,
taking the free simplicial abelian group it generates, and identifying the
subgroups generated by the basepoints to zero, this is (after using the Quillen
equivalence of simplicial pointed sets with pointed CW complexes as well as the Dold-Kan 
Quillen equivalence of
simplicial abelian groups with chain complexes of abelian groups concentrated in nonnegative degrees,  
and passing to the homotopy category)
equivalent to taking the pointed simplicial chain complex of a pointed simplicial complex; 
so, identifying $D^+(\mathbb{M}_0)$ with the bounded-below stable homotopy category, $g_*$ sends a spectrum to (the quasi-equivalence class of) 
its singular chain complex. Similarly, given a chain complex of abelian groups $C^\bullet$ concentrated
in nonnegative degrees,
using the Dold-Kan theorem to produce a simplicial abelian group from it, and forgetting the group
structure to get a simplicial pointed set, this process 
yields the generalized Eilenberg-Maclane spectrum
$HC^{\bullet}$; this describes $g^*$. As a result, we have
\begin{prop} The derived adjoint functors
\begin{eqnarray*} 
D^+(\mathbb{Z}) & \stackrel{g_*}{\longrightarrow} & D^+(\mathbb{M}_0),\\
D^+(\mathbb{Z}) & \stackrel{g^*}{\longleftarrow} & D^+(\mathbb{M}_0)\end{eqnarray*}
fit into the commutative diagram of triangulated categories
\[\xymatrix{ 
D^+(\mathbb{Z})\ar[r]^{g_*}\ar[d]^{\id} & 
 D^+(\mathbb{M}_0)\ar[r]^{g^*}\ar[d]^{\simeq} &
  D^+(\mathbb{Z})\ar[d]^{\id} \\
D^+(\mathbb{Z})\ar[r]^(.4){H} & 
 \Ho(\Mod(S)^+)\ar[r]^(.6){\sing} &
  D^+(\mathbb{Z}),}\]
where $\Ho(\Mod(S)^+)$ is the stable homotopy category of bounded-below spectra,
$H$ is the Eilenberg-Maclane functor, and $\sing$ is the singular chain complex functor.
\end{prop}

The derived adjunction of $g^*$ and $g_*$,
\[ \hyperExt^0_{\mathbb{Z}}(\sing(X)_\bullet,C_\bullet) \cong [X,HC_\bullet],\]
where $\hyperExt^0$ is the zeroth hyper-Ext functor, has as special cases the representability
of cohomology with coefficients in abelian group $A$ by the Eilenberg-Maclane
spectrum $HA$ (this happens when $C_{\bullet} \simeq A[0]$), and
the fact that the homotopy groups of the generalized Eilenberg-Maclane
spectrum of a complex $C_{\bullet}$ recover the homology groups of $C_{\bullet}$, that is,
$\pi_{*}(HC^\bullet) \cong H_*(C_{\bullet})$ (this happens when $X$ is homotopy equivalent to a sphere, 
since in that case the spectral sequence
\[ E_2^{*,*}\cong \Ext_{\mathbb{Z}}^*(H_{*}(S^n;\mathbb{Z}),H_*C_{\bullet})\Rightarrow 
\hyperExt_{\mathbb{Z}}^*(\sing(S^n)_\bullet,C_\bullet)\]
collapses on to a single line, $\Ext_{\mathbb{Z}}^0(H_n(S^n;\mathbb{Z}),H_*(C_\bullet))\cong
H_*(C_\bullet)$, at $E_2$).

It is worth making some remarks about whether or not
$\Spec \mathbb{Z}\stackrel{g}{\longrightarrow} \Spec \mathbb{M}_0$ is an effective descent morphism.
Given a pointed set $S$, the abelian group $\alpha S$ naturally has the structure map
$\alpha S\rightarrow \alpha \alpha^* \alpha S$ of a coalgebra
for the comonad $\alpha\alpha^*$. The question is whether the resulting functor
from pointed sets to $\alpha\alpha^*$-coalgebras is an equivalence of categories; if so, then $g$
is said to be an ``effective descent morphism'' (for this same story told in the classical
setting of schemes, see \cite{MR0255631}). By the weak Beck's theorem 
(see \cite{MR1712872}), one knows that 
$g$ is an effective descent morphism if and only if $\alpha$ preserves and reflects equalizers; it
is a very elementary exercise to verify that it indeed does. So $g$ is an effective descent morphism.
However, it is a much deeper matter to ask whether the derived adjunction
$g_*, g^*$ defines an effective descent morphism; naively, this would be asking whether
one has an equivalence of stable homotopy categories between bounded-below spectraand $g_*g^*$-coalgebras; or, in other words, can one recover a bounded-below spectrum 
$X$ from the coaction map $X\wedge H\mathbb{Z}\rightarrow X\wedge H\mathbb{Z}\wedge H\mathbb{Z}$.
The answer is yes for connective $X$, since this is enough information for one to write down the rest of the 
$H\mathbb{Z}$-Adams resolution of $X$, whose totalization recovers $X$.
So the question of whether $g_*$ and $g^*$ define a derived effective descent morphism,
this is really the question of what the $H\mathbb{Z}$-Adams spectral sequence converges to, i.e.,
identifying the $H\mathbb{Z}$-nilpotent completion functor. In general, questions about effective
descent for morphisms of derived categories are really questions about $E$-nilpotent completion
and generalized Adams spectral sequences. See \cite{hessdescent}
for more on these matters.

\section{Triangulated $\Spec \mathbb{Z}$ and $\Spec S$.}

Following Grothendieck, we work with categories of modules over 
rings or schemes, rather than the rings or schemes themselves; and furthermore, 
in this section, we want to explore the situation when we take the bounded-below derived category
$D^+(\mathbb{Z})$, and not simply the category of $\mathbb{Z}$-modules, as the real object
of interest when working with $\Spec \mathbb{Z}$, so in this section,
when we write $\Spec S$, we mean the homotopy category $\Ho(\Mod(S)^+)$ 
of connective $S$-modules (that is, $S$-modules whose homotopy
groups vanish in sufficiently low dimensions), of \cite{MR1417719}; 
when we write $\Spec \mathbb{Z}$, we have in mind the category 
$D^+(\mathbb{Z})$; and when we speak of a morphism
\begin{equation}\label{Spec S under Spec Z with triangulations} 
\Spec\mathbb{Z}\stackrel{f}{\longrightarrow} \Spec S, \end{equation}
we mean a pair of functors
\begin{eqnarray*} 
\Mod^+(S) & \stackrel{f^*}{\longrightarrow} & D^+(\mathbb{Z}),\\
\Mod^+(S) & \stackrel{f_*}{\longleftarrow} & D^+(\mathbb{Z}),\end{eqnarray*}
with $f_*$ fully faithful and 
with $f^*$ left adjoint to $f_*$ after passage to derived categories:
\[ \hyperExt^*_{\mathbb{Z}}(f^*(X),C_{\bullet}) \cong [ \Sigma^{-*}X,f_*C_{\bullet} ].\]
Here $\hyperExt$ is the hyper-Ext functor. Clearly we want $f_*$ to be the 
generalized Eilenberg-Maclane spectrum functor $H$, 
and as a consequence $f^*$ has to be the singular chain complex functor. Of course,
as a result of Shipley's theorem (see \cite{MR2306038}) identifying the
homotopy category of bounded-below chain complexes of abelian groups with the
homotopy category of bounded-below $H\mathbb{Z}$-modules, the functor $f^*$ should be 
thought of
as base change to $H\mathbb{Z}$---that is, the functor $X\mapsto X\wedge H\mathbb{Z}$.

Now we move on to Grothendieck duality.
Given a map of schemes $X\stackrel{f}{\longrightarrow} Y$, \linebreak ``Grothendieck duality 
for $f$'' is the existence of a functor 
$D^+(\mathcal{O}_Y)_{\qcoh}\stackrel{f^!}{\longrightarrow} D^+(\mathcal{O}_X)_{\qcoh}$
which is right adjoint to the functor $D^+(\mathcal{O}_X)_{\qcoh}\stackrel{Rf_*}
{\longrightarrow} D^+(\mathcal{O}_Y)_{\qcoh}$
induced by $f_*$ on the derived categories:
\[ Rf_*R\hom_{\mathcal{O}_X}(C_{\bullet}, f^!D_{\bullet})\cong 
R\hom_{\mathcal{O}_Y}(Rf_*C_{\bullet}, D_{\bullet}).\]
When $Y\cong \Spec k$ for some field $k$ and $f$ is finite type, the functor
$D^+(\mathcal{O}_X)\stackrel{I}{\longrightarrow} D^+(\mathcal{O}_X)^{\op}$ given by
\[ IC_{\bullet} = R\hom_{\mathcal{O}_X}(C_{\bullet},f^!(\mathcal{O}_Y[0])),\]
(where $\mathcal{O}_Y[0]$ is the complex consisting of $\mathcal{O}_Y$ in degree
zero and the trivial module in all other degrees) satisfies the isomorphism
$IIC_{\bullet}\cong C_{\bullet}$ for sufficiently small (``dualizable'') complexes
$C_{\bullet}$; this is the sense in which Grothendieck duality is really
a ``duality.''

We wish to describe such a situation for the morphism 
$\Spec \mathbb{Z}\stackrel{f}{\longrightarrow}\Spec S$, i.e., we want a
derived right adjoint for $f_*$. This takes the form of a functor
\[ \Ho(\Mod(S)^+)  \stackrel{f^!}{\longrightarrow}  D^+(\mathcal{Z})\] 
from the homotopy category of bounded-below spectra to the bounded-below derived
category of $\mathbb{Z}$-modules, such that we have the isomorphism
\[ \Ext^*_{\mathbb{Z}}(C_{\bullet},f^!X) \cong [\Sigma^{-*}HC_{\bullet}, X],\]
natural in the choice of chain complex $C_{\bullet}$ and spectrum $X$. 

That $f^!$ exists can be shown in more than one way: in \cite{MR1417719} it is
shown that $f^!$ actually exists as a functor from $S$-modules to $HZ$-modules, as 
a right adjoint to the forgetful functor, and $f^!(X)$ is isomorphic to 
the function spectrum $F(HZ,X)$; this fits with the definition of $f^!$ for a
finite morphism of schemes, in \cite{MR0222093}. Another method
for proving the existence of $f^!$ is Neeman's proof in \cite{MR1308405}, 
using the Brown representability
theorem for triangulated categories: since the Eilenberg-Maclane spectrum functor
$D(\mathbb{Z})\stackrel{H}{\longrightarrow} \Ho(S)$ 
is a functor defined on a compactly generated triangulated category, taking
values in a triangulated category, and which preserves coproducts, the existence
of a right adjoint for $H$ is automatic (and indeed, is even defined on the unbounded
derived category). Neeman used this argument to produce $f^!$ for a map of schemes $f$
but it works equally well to produce $f^!$ for the morphism
$\Spec \mathbb{Z}\stackrel{f}{\longrightarrow}\Spec S$. Of course, by the argument of the previous 
section, the functors $f_*$ and $f^*$ (which we regard as defining a morphism
$\Spec \mathbb{Z}\stackrel{f}{\longrightarrow} \Spec S$) coincide with the derived adjoint 
functors $g_*$ and $g^*$ between $D^+(\mathbb{Z})$ and $D^+(\mathbb{M}_0)$; so the 
Grothendieck duality we describe here for $f$ is equally well a Grothendieck duality
for $\Spec \mathbb{Z}\stackrel{g}{\longrightarrow}\Spec\mathbb{M}_0$.

\begin{prop} \label{homology of dualizing complex} 
Let $f$ be the morphism
\[ \Spec \mathbb{Z}\stackrel{f}{\longrightarrow}\Spec S ,\]
with the definition we have given immediately following 
morphism~\ref{Spec S under Spec Z with triangulations}.
Then the homology groups of the chain complex 
$f^!(S)$ are concentrated in a single degree. Specifically:
\begin{eqnarray*} H_i(f^!(S)) & \cong & [\Sigma^{i} H\mathbb{Z}, S] \\
  & \cong & \left\{ \begin{array}{ll} 
                      \Ext_{\mathbb{Z}}^1(\mathbb{Q},\mathbb{Z}) 
                           & \mbox{\ if\ } i=-1 \\
                      0    & \mbox{\ otherwise.} \end{array}\right. \end{eqnarray*}

The complex $f^!(S)$ has trivial homology in all degrees when pulled back to
any closed point of $\Spec \mathbb{Z}$. In other words,
if $j$ is the morphism
\[ \Spec \mathbb{F}_p\stackrel{j}{\longrightarrow}\Spec S ,\]
then 
\begin{eqnarray*} H_i(j^!(S)) & \cong & [\Sigma^{i} H\mathbb{F}_p, S] \\
  & \cong & 0.\end{eqnarray*} 
\end{prop}
\begin{proof}
Recall the theorem of Lin \cite{MR0402738}: for $A$ a
cyclic abelian group and $X$ a finite cell complex, we have the isomorphism
\[ [\Sigma^iHA,X] \cong \Ext^1_{\mathbb{Z}}(A\otimes_{\mathbb{Z}}\mathbb{Q},
 H_{i+1}(X;\mathbb{Z})).\]
Since $f^!$ is, by definition, a derived right adjoint for $f_* = H$, we have
the isomorphism
\[ \hyperExt_{\mathbb{Z}}^i(C_{\bullet},f^!S) \cong [\Sigma^{-i} HC_{\bullet}, S].\]
Let $\mathbb{Z}[0]$ be the complex with $\mathbb{Z}$ in degree $0$ and the trivial
abelian group in all other degrees. Then the spectral sequence
\[ \Ext_{\mathbb{Z}}^s(\mathbb{Z},H_tf^!S) \Rightarrow 
\hyperExt_{\mathbb{Z}}^{s-t}(\mathbb{Z}[0],f^!S)\]
collapses on to the $s=0$ line, so we get isomorphisms
\begin{eqnarray*}
 H_tf^!S & \cong & \hom_{\mathbb{Z}}(\mathbb{Z},H_tf^!S) \\
 & \cong & \hyperExt_{\mathbb{Z}}^{-t}(\mathbb{Z}[0],f^!S) \\
 & \cong & [\Sigma^{t}H\mathbb{Z},S],\end{eqnarray*}
which, by Lin's theorem, is $\Ext^1_{\mathbb{Z}}(\mathbb{Q},\mathbb{Z})$
when $t=-1$, and zero otherwise.

The spectral sequence 
\[ \Ext_{\mathbb{F}_p}^s(\mathbb{F}_p,H_tj^!S) \Rightarrow 
\hyperExt_{\mathbb{F}_p}^{s-t}(\mathbb{F}_p[0],j^!S)\]
also collapses on to the $s=0$ line, so we get isomorphisms
\begin{eqnarray*}
 H_tj^!S & \cong & \hom_{\mathbb{F}_p}(\mathbb{F}_p,H_tj^!S) \\
 & \cong & \hyperExt_{\mathbb{F}_p}^{-t}(\mathbb{F}_p[0],j^!S) \\
 & \cong & [\Sigma^tH\mathbb{F}_p,S],\end{eqnarray*}
which is zero for all $t$, again by Lin's theorem.
\end{proof}

We now give a description of $H_{-1}(f^!S)$ which is interesting for number-theoretic
reasons. If $K$ is a number field and $G$ is an algebraic group, one frequently encounters the groups
\[ \mathcal{O}\backslash G(\mathbb{A}_K) / G(K),\]
where $\mathbb{A}_K$ is the adele ring of $K$ (i.e., the restricted direct product of the 
completions of $K$ at all its places)
and $\mathcal{O}$ is a compact open subgroup of 
$G(\mathbb{A}_K^{\infty})$, the evaluation of $G$ at the
infinite adele ring of $K$. One recovers the complex points of
Shimura varieties in this way; in particular, when $G\cong GL_1$, the multiplicative
group scheme, and $\mathcal{O} = G(\mathbb{A}_K^{\infty})$, then
the group $\mathcal{O}\backslash G(\mathbb{A}_K) / G(K)$
is isomorphic to the ideal class group of $K$. We identify
the nonvanishing homology group $H_{-1}(f^!S)$ of the Grothendieck dualizing complex 
arising from the 
sphere spectrum as an additive analogue of the ideal class group:

\begin{prop}
Let $\mathcal{NF}$ be the category of finite extensions of $\mathbb{Q}$, and let $\Ab$ be the category
of abelian groups. Then the functor 
\begin{eqnarray*} \mathcal{NF} & \rightarrow & \Ab \\
K & \mapsto & G(\mathbb{A}_K^{\infty})\backslash G(\mathbb{A}_K) / G(K), \end{eqnarray*} 
where $G$ is the additive group scheme, $\mathbb{A}_K$ is the adele ring of $K$,
and $\mathbb{A}_K^{\infty}$ is the ring of infinite adeles of $K$ (i.e., 
the restricted direct product of the completions of the
$K$ at its Archimedean places),
is naturally isomorphic to the functor
\begin{eqnarray*} \mathcal{NF} & \rightarrow & \Ab \\
K & \mapsto & H_{-1}(f^!(M\mathcal{O}_K)), \end{eqnarray*} 
where $M\mathcal{O}_K$ is the Moore spectrum of the ring of integers $\mathcal{O}_K$ of $K$.
\end{prop}
\begin{proof}
From the proof of Prop.~\ref{homology of dualizing complex} we know that 
\begin{eqnarray*} H_{-1}(f^!(M\mathcal{O}_K)) & \cong & \Ext^1_{\mathbb{Z}}(\mathbb{Q},H_0(M\mathcal{O}_K;\mathbb{Z})) \\
 & \cong & \Ext^1_{\mathbb{Z}}(\mathbb{Q},\mathcal{O}_K),\end{eqnarray*}
and this isomorphism is natural in $K$; we need to show that
$\Ext^1_{\mathbb{Z}}(\mathbb{Q},\mathcal{O}_K)$ is naturally isomorphic to
$G(\mathbb{A}_K^{\infty})\backslash G(\mathbb{A}_K) / G(K)$, where $G$ is the additive
group scheme. Our method of proof is based on an unpublished proof of J. Michael Boardman's.
We have the short exact sequence of abelian groups
\[ 0\rightarrow \mathcal{O}_K\rightarrow K\rightarrow K/\mathcal{O}_K\rightarrow 0\]
and, after applying the functor $\Ext_{\mathbb{Z}}^1(\mathbb{Q},-)$ to this short
exact sequence, we get the short exact sequence
\[ 0\rightarrow
 \hom_{\mathbb{Z}}(\mathbb{Q},K)\rightarrow 
   \hom_{\mathbb{Z}}(\mathbb{Q},K/\mathcal{O}_K) \rightarrow
     \Ext_{\mathbb{Z}}^1(\mathbb{Q},\mathcal{O}_K) \rightarrow 0,\]
since $\hom_{\mathbb{Z}}(\mathbb{Q},\mathcal{O}_K)\cong \Ext^1_{\mathbb{Z}}(\mathbb{Q},K)\cong 0$.
Now the torsion group $K/\mathcal{O}_K$ decomposes as the direct sum 
$K\mathcal{O}_K\cong \bigoplus_{p}(K/\mathcal{O}_K)_{(p)}$ of its localizations at the primes of $\mathbb{Z}$,
so we have isomorphisms
\begin{eqnarray*} \hom_{\mathbb{Z}}(\mathbb{Q},K/\mathcal{O}_K) & \cong &
 \hom_{\mathbb{Z}}(\mathbb{Q},\bigoplus_p (K/\mathcal{O}_K)_{(p)} ) \\
 & \subseteq & \hom_{\mathbb{Z}}(\mathbb{Q},\prod_p (K/\mathcal{O}_K)_{(p)}) \\
 & \cong & \prod_p \hom_{\mathbb{Z}}(\mathbb{Q}, (K/\mathcal{O}_K)_{(p)}) \\
 & \cong & \prod_p \hom_{\mathbb{Z}}(\mathbb{Z}_{(p)}, (K/\mathcal{O}_K)_{(p)}) \\
 & \cong & \prod_p \hom_{\mathbb{Z}}\left(\colim \left(\mathbb{Z}\stackrel{p}{\longrightarrow} \mathbb{Z}\stackrel{p}{\longrightarrow} \dots\right),
     (K/\mathcal{O}_K)_{(p)} \right) \\
 & \cong & \prod_p \lim\left( \dots\stackrel{p}{\longrightarrow} 
              \hom_{\mathbb{Z}}(\mathbb{Z},(K/\mathcal{O}_K)_{(p)})\stackrel{p}{\longrightarrow} \hom_{\mathbb{Z}}(\mathbb{Z},(K/\mathcal{O}_K)_{(p)}) \right) \\
 & \cong & \prod_p \lim\left(\dots \stackrel{p}{\longrightarrow}  (K/\mathcal{O}_K)_{(p)}\stackrel{p}{\longrightarrow} (K/\mathcal{O}_K)_{(p)} \right) \\
 & \cong & \prod_p K_p,\end{eqnarray*}
where the isomorphism $\hom_{\mathbb{Z}}(\mathbb{Q}, (K/\mathcal{O}_K)_{(p)}) \cong \hom_{\mathbb{Z}}(\mathbb{Z}_{(p)}, (K/\mathcal{O}_K)_{(p)})$
holds because \linebreak $(K/\mathcal{O}_K)_{(p)}$ is uniquely $m$-divisible for every integer $m$ prime to $p$, so every morphism $\mathbb{Z}_{(p)}\rightarrow
(K/\mathcal{O}_K)_{(p)}$ extends uniquely to a morphism $\mathbb{Q}\rightarrow (K/\mathcal{O}_K)_{(p)}$. Morphisms $\mathbb{Q}\rightarrow (K/\mathcal{O}_K)_{(p)}$
which correspond to elements in $(\hat{\mathcal{O}}_K)_p\subseteq K_p$ are precisely those morphisms which factor through the quotient map 
$\mathbb{Q}\rightarrow \mathbb{Q}/\mathbb{Z}$, so we get Diagram 1, a commutative diagram with exact
rows.

\begin{sidewaysfigure}
\[\xymatrix{  & & \mbox{Diagram 1.} & & \\
0 \ar[r]\ar[d] & \hom_{\mathbb{Z}}(\mathbb{Q}/\mathbb{Z},(K/\mathcal{O}_K)_{(p)}) \ar[r]\ar[d]^{\cong} & 
                                  \hom_{\mathbb{Z}}(\mathbb{Q},(K/\mathcal{O}_K)_{(p)}) \ar[r]\ar[d]^{\cong} & 
                                       \hom_{\mathbb{Z}}(\mathbb{Z},(K/\mathcal{O}_K)_{(p)}) \ar[r] \ar[d]^{\cong} & 0\ar[d] \\
 0\ar[r]\ar[d] & \hom_{\mathbb{Z}}(\colim_n \mathbb{Z}/p^n\mathbb{Z},(K/\mathcal{O}_K)_{(p)}) \ar[r]\ar[d]^{\cong} &
                     \hom_{\mathbb{Z}}(\colim_n\mathbb{Z}, (K/\mathcal{O}_K)_{(p)}) \ar[r]\ar[d]^{\cong} & 
                         (K/\mathcal{O}_K)_{(p)}\ar[r]\ar[d]^{\cong} & 0\ar[d] \\
 0\ar[r]\ar[d] & \lim_n \hom_{\mathbb{Z}}(\mathbb{Z}/p^n,(K/\mathcal{O}_K)_{(p)}) \ar[r]\ar[d]^{\cong} & 
                      \lim \left( \dots \stackrel{p}{\longrightarrow}   (K/\mathcal{O}_K)_{(p)} \stackrel{p}{\longrightarrow}  (K/\mathcal{O}_K)_{(p)} \right) \ar[r]\ar[d]^{\cong} &
                          (K/\mathcal{O}_K)_{(p)}\ar[r]\ar[d]^{\cong} & 0\ar[d] \\
 0\ar[r] & (\hat{\mathcal{O}}_K)_p \ar[r] & K_p \ar[r] & (K/\mathcal{O}_K)_{(p)}\ar[r] & 0. }
\]
\end{sidewaysfigure}

Morphisms $\mathbb{Q}\rightarrow \prod_p (K/\mathcal{O}_K)_{(p)}$ which factor through the inclusion \linebreak
$\bigoplus_p (K/\mathcal{O}_K)_{(p)}\hookrightarrow \prod_p (K/\mathcal{O}_K)_{(p)}$ are the morphisms such that the composite of the $p$th component
map $\mathbb{Q}\rightarrow (K/\mathcal{O}_K)_{(p)}$ with the inclusion $\mathbb{Z}\hookrightarrow\mathbb{Q}$ is zero for all but finitely many $p$, i.e., 
these are the morphisms $\mathbb{Q}\rightarrow \prod_p (K/\mathcal{O}_K)_{(p)}$ such that the $p$th component map $\mathbb{Q}\rightarrow (K/\mathcal{O}_K)_{(p)}$
factors through $\mathbb{Q}\rightarrow\mathbb{Q}/\mathbb{Z}$ for all but finitely many $p$; these are the elements of
$\hom_{\mathbb{Z}}(\mathbb{Q},\prod_p (K/\mathcal{O}_K)_{(p)})\cong \prod_p K_p$ which correspond to elements of $\prod_p K_p$ whose $p$th component
is in $(\hat{\mathcal{O}}_K)_p\subseteq K_p$ for all but finitely many $p$. Hence we have the commutative diagram with exact rows:
\[ \xymatrix{ 0\ar[r]\ar[d] & \hom_{\mathbb{Z}}(\mathbb{Q},K)\ar[r]\ar[d]^{\cong} & \hom_{\mathbb{Z}}(\mathbb{Q},K/\mathcal{O}_K)\ar[r]\ar[d]^{\cong} &
 \Ext^1_{\mathbb{Z}}(\mathbb{Q},\mathcal{O}_K)\ar[r]\ar[d]^{\cong} & 0 \ar[d] \\
 0\ar[r] & K\ar[r] & \mathbb{A}^{fin}_K\ar[r] & \mathbb{A}^{fin}_K/K\ar[r] & 0,}\]
where by $\mathbb{A}^{fin}_K$ we mean the finite adeles of $K$, i.e., $\mathbb{A}^{fin}_K\cong \mathbb{A}_K/\mathbb{A}_K^{\infty}$.
\end{proof}

\section{$\Spec \mathbb{Z}$ and $\Spec S$, after To\"{e}n-Vaqui\'{e}.}

In this section we change gears entirely, and try a different approach to the construction
and study of a morphism $\Spec\mathbb{Z}\rightarrow\Spec S$: 
we will see what happens when we take $\Spec \mathbb{Z}$ to
mean the category of abelian groups, and we take $\Spec S$ to be the category
of connective $S$-modules, that is, the category of $S$-modules $X$ such that 
$\pi_i(X)$ is trivial for $i<0$. In this setup, unlike that of the previous section in
which we were really concerned with triangulated categories, we do not have a desuspension
operation defined on either of the categories under consideration; but there are some ways
in which this setup is a desirable one. For one thing, we will find that the base-change
functor from $\Spec S$ to $\Spec \mathbb{Z}$ is, in this setup, the zeroth homotopy
group functor $\pi_0$, rather than the singular chain complex functor as in the 
previous section; this is desirable because, in derived algebraic geometry, one would like
to think of a derived scheme $X^{\der}$ as consisting of an ordinary scheme $X$ together
with a homotopy-Cartesian presheaf $\mathcal{E}$ of $E_{\infty}$-ring spectra on the local Zariski 
(or perhaps
\'{e}tale, or syntomic, or...) site of $X$ with the property that $\pi_0(\mathcal{E})$ recovers
the structure ring sheaf of that site. From this point of view, we have a commutative
square
\[ \xymatrix{ X\ar[r]\ar[d] & X^{\der} \ar[d] \\
\Spec \mathbb{Z} \ar[r] & \Spec S}\]
which is a pullback square, in the sense that $X$ is obtained by applying $\pi_0$ to $X^{\der}$, 
i.e., $X$ is obtained by base-change along $\Spec \mathbb{Z}\rightarrow \Spec S$ (of course
this cannot be a pullback square in a categorical sense unless we actually describe a
category in which each of these objects lives; this also is
not a pullback square in the ``everything is triangulated''
setup of the previous section in this note). These observations we make
in this note are by no means a replacement for the well-developed theory of 
HAG, due to To\"{e}n and Vezzosi (primarily in \cite{MR2137288} and 
\cite{MR2394633}), and DAG, due to Lurie (in a continuing series, currently at
six volumes; the first is \cite{dagi}); 
for example, our situation (which involves only 
connective spectra, not periodic spectra) is still weaker than the technology required
to construct the derived moduli stack of derived elliptic curves, and the global sections of its
``$\infty$-sheafification'' (Joyal-Jardine fibrant replacement), 
(a localization of) the spectrum $tmf$ of topological modular forms.

There is one other sense in which the present setup (in which $\Spec \mathbb{Z}$ is taken as the
category of $\mathbb{Z}$-modules, and not some category of chain complexes of $\mathbb{Z}$-modules)
may be more desirable than one in which we are really concerned with triangulated categories:
the framework of To\"{e}n and Vaqui\'{e} for ``relative algebraic geometry,'' from the
paper \cite{MR2507727}, takes as input a complete, co-complete, closed symmetric monoidal category
$\mathcal{C}$, and produces a category of ``$\mathcal{C}$-schemes'' such that the
category $\mathcal{C}$ winds up being precisely the category of 
quasicoherent/Cartesian modules over the terminal object in the category of $\mathcal{C}$-schemes.
When applied to the category of abelian groups, with monoidal product the tensor product, 
To\"{e}n-Vaqui\'{e}'s framework produces the classical category of schemes; and when applied to
(a category Quillen-equivalent to)
the category of $S$-modules, with monoidal product the smash product, To\"{e}n-Vaqui\'{e}'s 
framework produces a useful category of ``derived schemes'' or ``spectral schemes.'' One can also
put the category of chain complexes of abelian groups through the machinery of To\"{e}n-Vaqui\'{e},
but then one gets a category of derived schemes over $\Spec H\mathbb{Z}$ (which was essentially the way we treated
$\Spec \mathbb{Z}$ in the previous section), 
not the classical category of schemes.

In most other ways the approach taken in the previous section is preferable to that of this
section: the connectivity assumption made on spectra in the previous section is
weaker, and more importantly, the adjunction of the previous section really is the 
``correct'' notion of a derived adjunction. 

First, we want to know that there really is a morphism
\[ \Spec \mathbb{Z}\stackrel{i}{\longrightarrow} \Spec S\]
in our present setup; in other words, we want to know that there is a faithful functor
\[ \Mod(\mathbb{Z})\stackrel{i_*}{\longrightarrow} \Mod(S)^{\geq 0}\]
where $\Mod(S)^{\geq 0}$ is the category of connective $S$-modules, and we want 
a functor 
\[ \Mod(\mathbb{Z})\stackrel{i^*}{\longleftarrow} \Mod(S)^{\geq 0}\]
which is well-defined on the homotopy category $\Ho(\Mod(S)^{\geq 0})$ and
such that we have a homotopy adjunction between them, i.e., an isomorphism
\[ \hom_{\Mod(\mathbb{Z})}( i^*X ,A) \cong [X,i_*A]\]
natural in the choice of spectrum $X$ and abelian group $A$. 
As usual we want
$i_*$ to be the Eilenberg-Maclane functor $H$, since it is the most obvious and natural embedding
of the category of abelian groups into the category of spectra. We claim that $i_*$ has
homotopy left adjoint $\pi_0$, the zeroth homotopy group functor. Indeed, we have the ``trivial'' 
Adams spectral sequence
\[ \Ext_{\grMod(\pi_*(S))}^s(\Sigma^t\pi_*(X),A) \Rightarrow [\Sigma^{t-s}X,HA],\]
where $\Ext_{\grMod(\pi_*(S))}$ is the derived functor of $\hom$ in the category of graded
modules over the graded ring $\pi_*(S)$ of stable homotopy groups of spheres.
Since $X$ and $HA$ are both connective, there is no room for differentials to come from or hit the
class in bidegree $(0,0)$, so we have isomorphisms
\begin{eqnarray*}  [X,HA] & \cong & \Ext_{\grMod(\pi_*(S))}^0(\pi_*(X),A) \\
 & \cong & \hom_{\grMod(\pi_*(S))}(\pi_*(X),A) \\
 & \cong & \hom_{\grMod(\pi_0(S))}(\pi_0(X),A) \\
 & \cong & \hom_{\Mod(\mathbb{Z})}(\pi_0(X),A).\end{eqnarray*}
So $\pi_0$ is homotopy left adjoint to $H$.

Now we proceed to Grothendieck duality. We want a homotopy right adjoint to $i_* = H$, i.e., we want
a functor $i^!$ with an isomorphism
\[ \hom_{\mathbb{Z}}(A,i^!X) \cong [HA,X]\]
natural in the choice of connective spectrum $X$ and abelian group $A$.
As a consequence of the results of the previous section, such a functor exists and is simply the zeroth
homology of the complex $f^!X$:
\[ i^!X \cong H_0f^!X,\]
and by Lin's theorem, $i^!$ is given on finite cell complexes $X$ by
\begin{eqnarray*} i^!X & \cong & \hom_{\mathbb{Z}}(\mathbb{Z},i^!X) \\
 & \cong & [H\mathbb{Z},X] \\
 & \cong & \Ext^1_{\mathbb{Z}}(\mathbb{Q},H_1(X,\mathbb{Z})).\end{eqnarray*}
Hence $i^!$ vanishes on finite cell complexes with finite first integral homology group $H_1$, e.g. the sphere spectrum.

\bibliography{/home/asalch/texmf/tex/salch}{}
\bibliographystyle{plain}
\end{document}